\documentclass[final]{siamltex}

\usepackage{epsfig}
\usepackage{amsmath,amsfonts,amssymb}
\usepackage{pict2e}
\usepackage{hyperref}
\def\setC{\,\mathbb{C}}
\def\setR{\mathbb{R}}
\def\setN{\mathbb{N}}

\def\als{{\scriptscriptstyle\#}}
\def\alp#1{{#1_\als}}

\newcommand{\R}{{\mathbb R}}
\newcommand{\C}{{\mathbb C}}

\newcommand{\Cnxn}{\C^{n \times n}}
\newcommand{\Rnxn}{\R^{n \times n}}
\newcommand{\set}[2]{ \left\{ {#1}\,\big\vert\, {#2}\right\}}

\newcommand{\inprod}[1]{\left\langle #1 \right\rangle}
\newcommand{\norm}[1]{\Vert #1 \Vert}
\newcommand{\pow}[2]{\!\left\langle #1 \right\rangle^{#2}\!}
\newcommand{\rchi}{\raise.35ex\hbox{$\chi$}}

\newtheorem{rumatemppu}{Example}
\newenvironment{example}{\begin{rumatemppu}\rm}{\end{rumatemppu}}
\newtheorem{rumatemppuu}{Algorithm}

\title{%POLYNOMIALS, 
FUNCTION THEORY OF ANTILINEAR OPERATORS 
} 
\author{
Marko Huhtanen\thanks{
Division of Mathematics,
Department of Electrical and Information Engineering,
University of Oulu,
90570 Oulu 57, 
Finland,
({\tt Marko.Huhtanen@aalto.fi}).}
\and 
Allan Per\"am\"aki\thanks{
Department of Mathematics and Systems Analysis,
Aalto University, 
Box 1100
FIN-02015,
Finland,
({\tt allan.peramaki@aalto.fi}).}
}

\begin{document}
\maketitle
\begin{abstract} Unlike in complex linear operator theory,
polynomials or, more generally, Laurent series 
in antilinear operators cannot be modelled with complex
analysis.
%appear much in the same way as polynomials in complex
%linear operators do. 
%Being somewhat different,
There exists a corresponding function space, though, 
surfacing in spectral mapping theorems. 
These spectral mapping theorems are inclusive in general. Equality
can be established in the self-adjoint case.
The arising functions
are shown to possess a biradial character.
It is shown that to any 
given set of Jacobi
parameters corresponds a  biradial measure
yielding these parameters in an iterative orthogonalization
process in this function space, once equipped with the corresponding $L^2$ structure. 
%In terms of Jacobi operators, 
%self-adjoint antilinear operators are connected
%with  biradial measure spaces.
%the $L^2$-theory on 
\end{abstract}


\pagestyle{myheadings}
\thispagestyle{plain}

\markboth{M. HUHTANEN AND A. PER\"AM\"AKI  
}{FUNCTIONS OF ANTILINEAR OPERATORS} 

\section{Introduction}  
Classical complex analysis is ubiquitous in complex linear Hilbert
space operator theory; see, e.g., \cite{Niko1,Niko2, AM} and references
therein. 
The interplay between
these two fields has found many forms
of which
the spectral mapping theorem is the most well-known and 
undoubtedly most widely used.
In real linear operator theory,  
complex linear operators 
constitute one extreme 
while antilinear operators constitute the other one.
There is strong evidence to expect 
that these two extremes turn out to be mathematically
equally rich \cite{HP,HS,HUNEVA,HP1}. 
In this paper, a function space structure is identified 
which has an analogous connection with antilinear operators as
complex analysis has with complex linear operators.
This function space surfaces in spectral mapping 
theorems for antilinear operators. 
An $L^2$ theory for it arises in the iterative construction
of invariant subspaces for self-adjoint antilinear operators 
once the appearing Jacobi operators are connected with biradial measure
spaces.

For the spectral mapping theorems,
suppose $A$ is a bounded antilinear
operator on a complex
%separable
Hilbert space $H$.
Antilinear means that $A\lambda=\overline{\lambda}A$ for
any complex number $\lambda$.
Then, for simplicity, take 
a polynomial $p(\lambda)=\sum_{k=0}^{j}\alpha_k\lambda^k$. 
If $x\in H$ is an eigenvector of $A$, there holds
\begin{equation}\label{eua}
p(A)x=\hat{p}(\lambda)x,
\end{equation}
where $p$ has transformed to  $\hat{p}(\lambda)=\sum_{k=0}^{\lfloor  \frac{j}{2} \rfloor }(\alpha_{2k} +
\alpha_{2k+1}\lambda)\left| \lambda\right|^{2k}.$
More generally, by allowing more complicated analytic functions $f$,
the transformed functions
can be expressed as
%\begin{equation}\label{funktio}
$$\hat{f}(\lambda)= 
 u(|\lambda|^2)+v(|\lambda|^2)\lambda.$$
%\end{equation}
with sufficiently regular complex valued functions $u$ and $v$.
These functions, denoted by $\mathcal{C}(r2)$,
constitute a vector space over $\setC$ carrying a natural notion
of product.  
The spectral mapping for the spectrum $\sigma(A)$ then takes the form
\begin{equation}\label{spektio}
\hat{f}(\sigma(A)) \subset \sigma(f(A)).
\end{equation}
The equality cannot be established in general.
However, if $A$ is additionally self-adjoint, the equality is shown
to hold in \eqref{spektio}. In particular, when treated antilinearly,
Hankel operators fit in this category in a natural way.

For the $L^2$ theory for these functions,
even if $x$ is not an eigenvector of $A$, collecting all 
the possible vectors on the left-hand side of \eqref{eua} gives
rise to an invariant subspace 
$$K(A;x)=\overline{\set{p(A)x}{p\in \mathcal{P}}}$$
of $A$, 
where $\mathcal{P}$ denotes the set of polynomials.
In the self-adjoint case, the function space $\mathcal{C}(r2)$ arises 
again once $A$ is represented on $K(A;x)$ with
an antilinear Jacobi operator
$$\alp{J} \tau, $$
where $\alp{J}$ is a complex symmetric (typically infinite) matrix
and $\tau$ denotes the conjugation operator. 
To characterize the Jacobi parameters on the diagonals of $\alp{J}$,
a biradially supported $L^2$ theory for $\mathcal{C}(r2)$ is devised. 
A curve in $\setC$ is said to be biradial if it intersects every origin
centred circle at most at two points. 
Namely, it is shown that to any 
given set of bounded Jacobi
parameters corresponds a compactly supported biradial measure
yielding these parameters in an iterative orthogonalization
process for polynomials $\hat{p}$ in the respective $L^2$ space. 
The case of unbounded
Jacobi parameters can be treated in terms of conditions on the moments
recorded in a Hankel-like matrix. 
Regarding the lack of uniqueness of this correspondence, the finite dimensional
case is completely solved.

The paper is organized as follows. In Section \ref{sec2}
spectral mapping theorems for antilinear operators are derived.
The corresponding function space structure is identified. 
In Section \ref{sec3} a theory for self-adjoint 
antilinear Jacobi operator is developed. To deal
with the Jacobi parameter problem, biradial $L^2$
spaces are introduced. In Section \ref{sec4}
the unbounded case is considered.
In Section \ref{noninjfinite}
the non-injective determination of the Jacobi 
parameters is solved in finite dimensional cases.

\section{Functions of  antilinear operators 
and spectral theory}\label{sec2}  
A continuous additive\footnote{An operator $B$ on $H$ is additive if
$B(x+y)=Bx+By$ for any $x,y\in H$.} 
operator on a complex
%separable
Hilbert space $H$ is real linear, i.e., 
it commutes with real scalars.
This fact can be found
already in Banach's classic book on linear operators 
\cite{BA}.\footnote{It is instructive to 
bear in mind that Banach's book started with additive operators
and dealt only with real scalars,
a fact which sometimes was regarded as curious \cite[p.397]{Rutiini}.} 
Denote the family of such operators
by $\mathcal{B}(H)$.
The norm of $B \in \mathcal{B}(H)$
is defined as
%\begin{equation}\label{norm}
$$||B||=\sup_{||x||=1}||Bx||.$$
%\end{equation}
The adjoint of $B$ is the real linear operator $B^*$ %\in \mathcal{B}(H)$
satisfying 
$$(Bx,y)_{\setR}=(x,B^*y)_{\setR}$$
for every $x,y \in H$. Here $(\cdot,\cdot)_{\setR}={\rm Re}\,(\cdot,\cdot)$,
where $(\cdot,\cdot)$ denotes the inner product on $H$.
If $B^*=B$, then $B$ is said to be self-adjoint.

There exists a
unique separation of $B$
into its complex linear
and antilinear parts as 
$$B=C+A,$$ where
\begin{equation}\label{sepa}
C=\frac{1}{2}(B-iBi) \, \mbox{ and }\,A=
\frac{1}{2}(B+iBi).
\end{equation}
That is, then $C$ is complex linear while $A$ is antilinear, i.e., 
$A\lambda=\overline{\lambda}A$ holds for any $\lambda\in \setC$.

Of course, the case of complex linear operators, i.e., when $A=0$, 
has been extensively studied. In what follows, we are more interested in the antilinear case, i.e., when $C=0$.
To this end, the following space of polynomials  was introduced in \cite{HP}.

\begin{definition}\label{defpol}
Polynomials of the form
\begin{equation}\label{polap}
\hat{p}(\lambda)=\sum_{k=0}^{\lfloor  \frac{j}{2} \rfloor }(\alpha_{2k} +
\alpha_{2k+1}\lambda)\left| \lambda\right|^{2k}
\end{equation}
with $\alpha_k\in \setC$ and, for $j$ even $\alpha_{j+1}=0$,
are denoted by $\mathcal{P}_j(r2)$. Their union 
$\cup_{j=0}^\infty \mathcal{P}_j(r2)$ is 
denoted by $\mathcal{P}(r2).$
\end{definition}

For a nonzero polynomial, the greatest integer $l$ such that
$\alpha_l\not=0$ in \eqref{polap}
is called the degree of the polynomial.

To see how these relate with real linear operators,
take a standard analytic polynomial 
$p(\lambda)=\sum_{k=0}^{j}\alpha_k\lambda^k$. 
If $B\in \mathcal{B}(H)$, then 
$p(B)$ is defined as $$p(B)=\sum_{k=0}^{j}\alpha_kB^k.$$
Of course, polynomials in complex linear operators appear regularly. 
Polynomials in antilinear operators
have  recently proved  useful 
in numerical linear
algebra and approximation theory
\cite{EHVP, HP1,HP}. 
A first notable difference between 
these two extremes is the fact that a polynomial
in an antilinear operator typically becomes genuinely a real linear
operator, i.e., it has nonzero complex and antilinear  parts.
Second, $\mathcal{P}_j(r2)$ is a natural function space 
to analyze polynomials in antilinear operators.

Namely, associate with $p$ the polynomial $\hat{p}$ defined in \eqref{polap}. 
At the most fundamental level, such a transformation
occurs with  a spectral mapping theorem
for antilinear operators. This is
readily seen in terms of eigenvalues. If 
$A$ is antilinear and  $Ax=\lambda x$ for a nonzero $x\in H$ and $\lambda \in \setC$,
then \eqref{eua} holds. Consequently, 
the eigenvalues of $A$ are mapped with $\hat{p}$ to
be among the eigenvalues of $p(A)$.

Definition \ref{defpol}  was not aimed at maximal generality, however.
First, aside from eigenvalues, we have a more general notion
of spectrum.
Second, Laurent series in a complex linear operator can be defined as long
as the spectrum is contained in the annulus of convergence.
The spectrum of a real linear operator is defined in a natural way
as follows.

\begin{definition} 
The spectrum of $B\in \mathcal{B}(H)$ consists
of those points $\lambda \in \setC$ for which $\lambda I-B$
is not boundedly invertible. The set of these points is denoted
by $\sigma(B)$.
\end{definition}

The spectrum of a real linear operator 
is always compact. For antilinear operators, it is circularly symmetric
with respect to the origin, and it can be empty as well; see \cite{HS}. Hence,
a lack of spectral radius means that
spectrum alone is not sufficient to determine convergence of Laurent series
in real linear operators.
Still,
if $f$ is analytic in an annulus centred at the origin with
the Laurent series
\begin{equation}\label{lau}
f(\lambda)=\sum_{k=-\infty}^{\infty}\alpha_k \lambda^k,
\end{equation}
then, under appropriate assumptions on an antilinear $A$,
the real linear operator 
$f(A)$ is well defined in terms of the series.\footnote{Form
 $\limsup_{j \rightarrow \pm \infty}|\alpha_j|^{-1/j}$ and compare with 
$\limsup_{j \rightarrow \pm \infty}||A^j||^{1/j}$.} 
(If $\alpha_k=0$ for $k=-1,-2,\ldots$, then we have
a disc centred at the origin.) We assume that the spectrum
$\sigma(A)$ is contained in the annulus of convergence.
By inspecting  eigenvalues,
$f$ gets transformed in the process 
completely analogously as
\begin{equation}\label{functjoukko}
\hat{f}(\lambda)= 
%\sum_{k=-1}^{-\infty}(
%\alpha_{2k+1}\frac{\left| \lambda\right|^{2}}{\overline{\lambda}}
%+
%\alpha_{2k})\left| \lambda\right|^{2k}+
\sum_{k=-\infty}^{\infty}(\alpha_{2k} +
\alpha_{2k+1}\lambda)\left| \lambda\right|^{2k}= 
 u(|\lambda|^2)+v(|\lambda|^2)\lambda.
\end{equation}
Clearly, $u$ and $v$ are polynomials if and only if $f$ is.
Observe that $\hat{f}$ maps origin centred circles to
circles. In particular,
we regard these functions as biradial 
%functions of the type appearing on the right-hand side
for the following reason.

\smallskip

\begin{example}
The functions $u$ and $v$ in \eqref{functjoukko}
are ``biradially'' uniquely determined 
in the
following sense. Suppose $\theta_1,\theta_2\in [0,2\pi)$
and $\theta_1 \not=\theta_2$.
Then, for $r$ fixed, the conditions
\begin{equation}
\left\{ \begin{array}{lll}
u(r^2)+v(r^2)re^{\theta_1}&=&a(r)\\
u(r^2)+v(r^2)re^{\theta_2}&=&b(r)
\end{array} \right.
\end{equation}
with $a(r)$ and $b(r)$ given, determine the values
of $u(r^2)$ and $v(r^2)$.\footnote{This makes extending  $\mathcal{C}(r2)$
apparent. To determine the values $2k$-radially, 
we are lead to consider functions
of the form $\sum_{j=0}^{2k-1}u_j(|\lambda|^{2k})\lambda^j$,
where $u_j$ are sufficiently  smooth and $k\in \setN$.} 
Of course, if $v\equiv 0$, then we are dealing a standard
radial function.
\end{example}

\smallskip

Denote by $\mathcal{C}(r2)$ functions of the form on
right-hand side of \eqref{functjoukko}.
Assuming that the antilinear $A$ is additionally self-adjoint,
we shall allow
%more general 
continuous functions of this type %in $\mathcal{C}(r2)$ 
by noting that $A^2$ is
complex linear positive semidefinite.
To ensure that this is consistent, let us invoke  
the following proposition.
\begin{proposition}[\cite{HS}]\label{HSprop}
Let $A\in\mathcal{B}(H)$ be antilinear. Then
$\lambda \in \sigma(A)$ if and only if $|\lambda|^2\in\sigma(A^2)$.
\end{proposition}

\begin{definition}
Let $A\in\mathcal{B}(H)$ be antilinear and self-adjoint.
Let $u$ and $v$ be complex valued continuous functions defined
on the compact subset $\sigma(A^2)\subset[0,\infty)$. Then the function
$\hat{f}:\sigma(A)\to\C$ defined by
\begin{equation}\label{biradf}
\hat{f}(\lambda) = u(|\lambda|^2) + v(|\lambda|^2)\lambda
\end{equation}
is called a continuous biradial function.  %For $\hat{f}$ set 
Define
\[f(A) = u(A^2) + v(A^2)A.\]
\end{definition}

Clearly, $\mathcal{C}(r2)$ is a vector space over $\setC$. 
There exists a natural notion of (noncommutative) product as well.
To this end, 
consider again \eqref{eua}.
For two elements 
$\hat{f}(\lambda)= u_1(|\lambda|^2)+v_1(|\lambda|^2)\lambda$
and $\hat{g}(\lambda)= u_2(|\lambda|^2)+v_2(|\lambda|^2)\lambda$
%$\hat{f}= u_1+\lambda v_1$
%and $\hat{g}= u_2+\lambda v_2$
of $\mathcal{C}(r2)$, we have for the eigenvalues
\begin{equation}\label{ulo}
f(A)g(A)x=\hat{h}(\lambda)x=(\hat{f}*\hat{g})(\lambda)x
\end{equation}
once we define
$$\hat{h}(\lambda)=
%(u_1(|\lambda|^2)+\lambda v_1(|\lambda|^2)\tau)
%(u_2(|\lambda|^2)+\lambda v_2(|\lambda|^2))
u_1(|\lambda|^2)u_2(|\lambda|^2)
+|\lambda|^2 v_1(|\lambda|^2)\overline{v_2(|\lambda|^2)}+
\left(u_1(|\lambda|^2)v_2(|\lambda|^2)+
\overline{u_2(|\lambda|^2)}v_2(|\lambda|^2)\right)\lambda.
$$
Or, in short 
\begin{equation}\label{tulocr2}
\hat{f}*\hat{g}=(u_1+v_1\lambda\tau)
(u_2+v_2\lambda),
\end{equation}
 where $\tau$ denotes the conjugation operator.
%This product is noncommutative.

After identifying an appropriate function space %$\mathcal{C}(r2)$
for antilinear operators, we are concerned with  
having a spectral mapping theorem. To this
end we need the following lemma.
%We shall use the following lemma, but eventually obtain the sharp
%result in Corollary \ref{spectradprop}.

\begin{lemma}\label{sapolyapprox}
Let $A\in\mathcal{B}(H)$ be antilinear and self-adjoint, and
let $\hat{f}:\sigma(A)\to\C$ be a continuous biradial function. Then
for any $\epsilon>0$ there exists a polynomial $p$ such that
\[\norm{f(A) - p(A)} < \epsilon
\qquad\text{and}\qquad
\max_{\lambda\in\sigma(A)} |\hat{f}(\lambda) - \hat{p}(\lambda)|<\epsilon.
\]
\end{lemma}

\begin{proof}
The result follows by applying the Weierstrass approximation
theorem to $u$ and $v$ in \eqref{biradf} on the compact set $\sigma(A^2)$
and using the continuous function theory for the complex linear $A^2$.
\end{proof}

%% Of course, we are not merely interested in eigenvalues.
%% The spectrum of a real linear operator is defined in a natural way
%% as follows.

%% \begin{definition} 
%% The spectrum of $B\in \mathcal{B}(H)$ consists
%% of those points $\lambda \in \setC$ for which $\lambda I-B$
%% is not boundedly invertible.
%% \end{definition}

%% The spectrum of a real linear operator 
%% is always compact. It can be empty as well; see \cite{HS}. 

%% The function space $\mathcal{C}(r2)$ appears in the spectral
%% mapping theorem for antilinear operators.

\begin{theorem}\label{spectmapthm1}
Let $A\in \mathcal{B}(H)$ be antilinear.
Assuming $f(A)$ is defined for $f$ in \eqref{lau}, there holds
$$\hat{f}(\sigma(A)) \subset \sigma(f(A)).$$
Moreover, if $A$ is self-adjoint, then this further holds for
all continuous biradial functions $\hat{f}:\sigma(A)\to\C$ in
\eqref{biradf}.
\end{theorem}

\begin{proof} For any $B\in \mathcal{B}(H)$ holds
\begin{equation}\label{aprop}
\sigma_a(B)\cup \overline{\sigma_a(B^*)}=\sigma(B),
\end{equation}
 where
$\sigma_a(B)$ denotes the approximate point spectrum of $B$ and the
bar denotes complex conjugation (not closure).

Take $\lambda \in \sigma_a(A)$. Then for any given $\epsilon>0$ there
exists a unit vector $x\in H$ such that 
%$||(\lambda I -A)x||\leq\epsilon$, i.e., 
$Ax=\lambda x +v$ with $||v||\leq \epsilon$. Let $\theta \in \setR$.
Since $Ae^{i\theta}x=e^{-i\theta}\lambda x+e^{-i\theta}v$, we may conclude that
$\sigma_a(A)$ is circularly symmetric with respect to the origin.
Therefore by  \eqref{aprop}, also $\sigma(A)$ is 
circularly symmetric with respect to the origin.

Consider a rational function
$r(\lambda)=\sum_{k=-l}^{j}\alpha_k \lambda^k$ with $l,j\in \setN$.
(For this the identity \eqref{eua} holds analogously.)
By continuity, we have $r(A)x=\hat{r}(\lambda)x +w$, where $||w||$
can be made arbitrarily small by choosing $\epsilon$ small enough.
Consequently, $\hat{r}(\lambda)\in \sigma_a(r(A))$
whenever $\lambda\in \sigma_a(A)$. Regarding $f$, we have
$||f(A)-r(A)||\leq\hat{\epsilon}$ and
$\max_{\mu\in\sigma(A)} |\hat{f}(\mu) - \hat{r}(\mu)| \leq\hat{\epsilon}$
for any $\hat{\epsilon}>0$ by choosing
$l$ and $j$ large enough. Thereby 
\begin{align*}
\norm{f(A)x - \hat{f}(\lambda)x}
&\leq \norm{f(A)x - r(A)x} + \norm{r(A)x - \hat{r}(\lambda)x}
+ \norm{\hat{r}(\lambda)x - \hat{f}(\lambda)x} \\
&\leq 2\hat{\epsilon} + \norm{w},
\end{align*}
%$$||f(A)x-r(A)x||=
%||f(A)x-
%\hat{r}(\lambda)x +w||
%<\hat{\epsilon},$$
%so that by taking the limit we have $\hat{f}(\lambda)\in \sigma_a(f(A)).$
and it follows that $\hat{f}(\lambda)\in \sigma_a(f(A)).$
The case with self-adjoint $A$ is similar, only then we have a
polynomial in place of $r$
obtained by applying Lemma \ref{sapolyapprox}.

Similarly can be dealt with the case 
$\lambda \in \overline{\sigma_a(A^*)}$.
To see this, it suffices to consider the case of an eigenvalue of $A^*$
and the corresponding circle of radius $|\lambda|$ centred
at the origin. 
The claim follows from 
$r(A)^*=\sum_{k=-l}^j A^{*k}\overline{\alpha_k}=
\sum_{k=-\lfloor \frac{l}{2} \rfloor}^{\lfloor  \frac{j}{2} \rfloor }(\overline{\alpha_{2k}} +
\alpha_{2k+1}A^*)A^{*2k}$ after multiplying by the corresponding eigenvector.
\end{proof}

In general, the equality does not hold. For example, if the spectrum
of $A$ is empty, then with $\hat{f}(\lambda)=|\lambda|^2$ we have
$\hat{f}(\sigma(A)) = \emptyset$, but $\sigma(A^2)\not=\emptyset$.
The equality can be established in the following case,
whose proof will be postponed
until the end of Section \ref{infjacobisec}. 
%The theorem can also be proved by first
%considering it in the diagonalizable case and then using approximation
%results from \cite{Santtu}.

\begin{theorem}\label{spectmapthm2}
Assume $A\in \mathcal{B}(H)$ is antilinear
and self-adjoint. Then, for all continuous biradial functions
$\hat{f}:\sigma(A)\to\C$ in \eqref{biradf},
\begin{equation}\label{spectmapeq}
\hat{f}(\sigma(A)) = \sigma(f(A)).
\end{equation}
\end{theorem}

%% \begin{proof} An operator $B\in \mathcal{B}(H)$ is said to be orthogonally 
%% diagonalizable if there exists an orthonormal basis
%% $\{e_j\}$ of $H$ such that $B ze_j=
%% (\alpha_j+\beta_j\overline{z})e_j$
%% with fixed $\alpha_j,\beta_j\in\setC$ and for any $z \in\setC$.
%% For an antilinear self-adjoint operator $A$ we have a split
%% $A=A_1+A_2$ such that $A_1$ and $A_2$ are orthogonally 
%% diagonalizable and  compact antilinear self-adjoint operators,
%% respectively. Clearly, the claim holds for $A_1$.
%% Moreover, in this split the norm of $A_2$ can be
%% made arbitrarily small \cite{Santtu}. Thereby, by continuity,
%% the claim holds also for $A$.
%% \end{proof}

%% As the proof shows, when combined with the results of \cite{Santtu},
%% functions \eqref{functjoukko} provide a nearly perfect model for
%% antilinear self-adjoint operators.

%It follows from \cite[Proposition 2.15]{HS} that
%the norm of an antilinear self-adjoint $A\in\mathcal{B}(H)$
%equals its spectral radius,
%$\norm{A} = \sup_{\lambda\in\sigma(A)} |\lambda|$.
%The following generalization is proved at the end of Section
%\ref{infjacobisec}.

\begin{corollary}\label{spectradprop}
Assume $A\in \mathcal{B}(H)$ is antilinear
and self-adjoint. Then
\[\norm{f(A)} = \max_{\lambda\in\sigma(A)} |\hat{f}(\lambda)|
= \max_{\lambda\in\sigma(f(A))} |\lambda|.\]
\end{corollary}

For an illustration, 
Hankel operators yield an immediate nontrivial family of 
self-adjoint antilinear operators.

\smallskip

\begin{example}\label{hankel} 
Hankel operators \cite{PEL} constitute a
natural  family of self-adjoint antilinear  operators once
treated as follows. (See also \cite[Sec. 7 and 8]{PUTI}.)
Denote
by $H^2(\mathbf{D})$ the Hardy space, where $\mathbf{D}$ is
the unit disc.  Let $a\in L^{\infty}(\mathbf{T})$, 
where $\mathbf{T}$
denotes the unit circle and $P:L^2(\mathbf{T})\rightarrow H^2(\mathbf{D})$
is the orthogonal projector onto $H^2(\mathbf{D})$.\footnote{Recall
that  $H^2(\mathbf{D})$ can be identified with those elements
of $L^2(\mathbf{T})$ which have vanishing Fourier coefficients
for negative indices.}
Then  define
$$g\longmapsto PM_a\overline{g}$$
on $H^2(\mathbf{D})$, 
where $M_a$ is the multiplication operator $g \mapsto ag$.
Represented on $l^2$, we have $H\tau$, where $H$
is an  infinite Hankel matrix.\footnote{This antilinear treatment
leads to different problems. For instance,
unitary similarity for  $H\tau$ becomes
unitary consimilarity for $H$.
We are only aware of the unitary similarity problem for $H$
\cite{MPT}.}
\end{example}

\smallskip

Like Hankel operators, self-adjoint antilinear operators have 
also been studied in a complex linear setting
more generally; see \cite{GP,GPII} where also many examples 
are given. 

We do not know how to completely characterize those antilinear operators
for which the spectral mapping Theorem \ref{spectmapthm1}
holds with 
equality. However, we {\it conjecture}
that these are precisely those antilinear $A$ for which
$\sigma(A^2) \subset [0,\infty)$. Necessity follows by
taking $f(\lambda)=\lambda^2$ and using Proposition \ref{HSprop}.
In finite dimensional $H$, sufficiency is established by the notion
of contriangularizability, which is equivalent to
$\sigma(A^2) \subset [0,\infty)$ \cite{HJ1}.
%It is a natural problem, though.
%In the finite
%dimensional case there is the notion of condiagonalizability
%which yields a sufficient condition. (For the probability
%of being condiagonalizable, see
%\cite{HP}.)
(For the probability
of being contriangularizable, see
\cite{HP}.)

\smallskip

\begin{example} Suppose an antilinear $A\in \mathcal{B}(H)$ is such that
$\sigma(A^2) \subset [0,\infty)$.
Then also Gelfand's formula 
$$\lim_{j \rightarrow \infty} ||A^j||^{1/j}= \max_{\lambda \in \sigma(A)}|\lambda |$$
holds.
\end{example}

\section{Antilinear Jacobi operators
and $L^2$ theory for $\mathcal{C}(r2)$}\label{sec3}

Like in the complex linear case,
the spectrum is intimately related with the notion of invariant subspace.
That is, if an antilinear operator has an eigenvector, then
its span yields an invariant subspace. By invariance is meant the 
following. 

\begin{definition} 
A subspace $K$ of $H$ is said to be invariant for an
operator $B\in \mathcal{B}(H)$ 
if $BK\subset K$.
\end{definition}

For an antilinear
$A\in \mathcal{B}(H)$  and $x \in H$, 
form the invariant subspace
\begin{equation}\label{krylov}
K(A;x)=\overline{\set{p(A)x}{p\in \mathcal{P}}}, 
\end{equation}
where $\mathcal{P}$ denotes the set of polynomials.
If $x$ is an eigenvector, then $K(A;x)$ is one dimensional. For the other
extreme,
if $K(A;x)=H$,
then $x$ is said to be a cyclic vector for $A$, like
in the complex linear case.

An orthonormal basis of $K(A;x)$ can be generated analogously
to the finite dimensional case described in \cite{EHVP}.
If the dimension of \eqref{krylov} is infinite, 
then $A$ is  represented as
$$\alp{H}\tau :l^2(\setN) \rightarrow
l^2(\setN),$$
where the (infinite) matrix $\alp{H}$ is of 
Hessenberg type with real subdiagonal
entries and $\tau$ denotes
the standard conjugation operation on $l^2(\setN)$. If
the dimension of \eqref{krylov} is finite, say $n$,
replace $l^2(\setN)$ with $\setC^n$.
In case $A$ is self-adjoint, $\alp{H}$ is a tridiagonal complex symmetric Jacobi matrix.
In this way complex Jacobi operators arise naturally in antilinear
operator theory.

\begin{proposition}\label{ortprop}
Let $A\in\mathcal{B}(H)$ be antilinear and self-adjoint,
and let $x,y\in H$. If $y \perp K(A;x)$, then $K(A;y) \perp K(A;x)$.
\end{proposition}

\begin{proof}
Let $p$ and $q$ be polynomials, and denote
$q(\lambda) = u(\lambda^2) + v(\lambda^2)\lambda$, where $u$ and
$v$ are polynomials. Then
\begin{align*}
(p(A)x,q(A)y) &= (p(A)x,u(A^2)y) + (p(A)x,v(A^2)Ay) \\
&=(\overline{u}(A^2)p(A)x,y) + \overline{(A\overline{v}(A^2)p(A)x,y)} \\
&=(\widetilde{p}(A)x,y) + \overline{(\widetilde{q}(A)x,y)} = 0,
\end{align*}
where we have used the fact that $A^2$ is $\C$-linear self-adjoint
and $\widetilde{p},\,\widetilde{q}$ are some polynomials.
\end{proof}

By Proposition \ref{ortprop},
given a bounded self-adjoint antilinear operator $A$ on $H$,
we can express $H$ as an orthogonal direct sum of invariant subspaces
\begin{equation}\label{hdirectsum}
H = \bigoplus_{\alpha} K(A;x_\alpha).
\end{equation}
We denote the restriction of $A$ to these subspaces by
\begin{equation}\label{adirectsum}
A_\alpha = A|_{K(A;x_\alpha)}.
\end{equation}
We have
\begin{equation}\label{spectcup}
\sigma(A) = \overline{\bigcup_\alpha \sigma(A_\alpha)},
\end{equation}
since $\sigma(A^2) = \overline{\bigcup_\alpha \sigma(A_\alpha^2)}$,
and we apply Proposition \ref{HSprop} noting that
$\sigma(A)$ and $\sigma(A_\alpha)$ are circularly symmetric
with respect to the origin.

With these preliminaries, next we show that
there exists a correspondence between self-adjoint antilinear Jacobi operators
and positive measures on the plane, much as in complex
linear operator theory. The construction relies on $L^2$ theory
for $\mathcal{C}(r2)$ by orthogonalizing polynomials \eqref{polap}
supported on these measure spaces.
Since the invariant subspaces $K(A;x)$ may be finite or
infinite dimensional, we shall deal with them separately, starting with the
finite dimensional case. Then, once we have handled the infinite dimensional
case, Theorem \ref{spectmapthm2} will be proved.

It is instructive to bear in mind that, as described in Example
\ref{hankel},  everything applies to Hankel operators.

\subsection{Biradial measures with finite support and antilinear Jacobi operators on 
finite dimensional spaces}\label{bifisu}
In what follows it is shown, by partly following \cite{HP},
 that there exists a correspondence
between (discrete) biradial measures and antilinear Jacobi
operators.
For convenience,  the following notation for
the monomials appearing in Definition \ref{defpol} is employed.
\begin{definition}\label{polynot}
Let $k \in \setN$. Denote by $\pow{\lambda}{k}$ the monomials defined by
$$\pow{\lambda}{k} =
\begin{cases}
\lambda |\lambda|^{2j}, &\text{if $k$ is odd and $k=2j+1$}, \\
|\lambda|^k, &\text{if $k$ is even}.
\end{cases}$$
\linebreak
\end{definition}

To put it short, biradiality means the second item in the following definition.

\begin{definition}\label{finitemeas}
Assume $n\in \setN$.  Let
$\rho_k > 0,\,\lambda_k\in\C \quad(k=1,\dots,n)$ be such that
\begin{enumerate}
\item[(i)] $\sum_{k=1}^n \rho_k = 1$,
\item[(ii)] $\lambda_1,\dots,\lambda_n$ are distinct
and any origin centred circle intersects at most two of them.
\end{enumerate}
Furthermore, the complex numbers $\lambda_1,\dots,\lambda_n$ 
are assumed to be ordered such that
\begin{alignat*}{2}
|\lambda_{2k-1}| &= |\lambda_{2k}| && \text{for } k = 1,\dots,m, \\
|\lambda_{2k-1}| &< |\lambda_{2k+1}| && \text{for } k = 1,\dots,m-1, \\
|\lambda_{2m+1}| < |\lambda_{2m+2}| &< \cdots < |\lambda_n|
\end{alignat*}
for $m\in \setN$.
Then a positive measure on the plane defined by
\begin{equation}
\rho = \sum_{k=1}^n \rho_k \delta_{\lambda_k}
\end{equation}
is called a biradial measure with finite support.
\end{definition}

Assume $\rho$ is a biradial measure with finite support. 
On $\mathcal{P}(r2)$ let us use the $L^2$ inner product 
\begin{equation}\label{innerp}
\inprod{p,q} = \int_{\C} p\overline{q}\,d\rho =
\sum_{k=1}^n p(\lambda_k) \overline{q(\lambda_k)} \rho_k.
\end{equation}
Consider
the monomials $1,$ $\pow{\lambda}{1},$ $\pow{\lambda}{2},$ $\dots,$
$\pow{\lambda}{n-1}$. 
Executing the Gram-Schmidt orthogonalization process 
with respect to this
the inner product  yields an orthonormal sequence
of polynomials $p_0,p_1,\dots,p_{n-1}$ such that each $p_j$ has degree $j$.
This is a consequence of the following proposition  guaranteeing
that $\norm{p}^2 = \inprod{p,p} > 0$ for all $p \in \mathcal{P}(r2)$
with degree less than $n$.
\begin{proposition}[\cite{HP}]\label{zeroprop}
Let $p \in \mathcal{P}_d(r2)$ be nonzero. The following claims hold:
\begin{enumerate}
\item\label{zerit1} If $p$ has two distinct zeroes of the same modulus,
then all numbers of that modulus are zeroes.
\item Let $m$ be the number of nonzero moduli for which all numbers of
that modulus are zeroes and let $s$ be the number of moduli for which
exactly one number is a zero. Then $2m+s \leq d$.
\end{enumerate}
\end{proposition}

Consider the product \eqref{tulocr2}.
Expressing the polynomial 
$\lambda\tau p_j(\lambda)=\lambda \overline{p_j(\lambda)}$ as
a linear combination of $p_0,\dots,p_{j+1}$ by imposing
orthogonality gives rise to the three term recurrence
\begin{equation}
\label{psrecurs}
\beta_{j+1} p_{j+1}(\lambda) = \lambda\overline{p_j(\lambda)} - \alpha_{j+1}
p_j(\lambda) - \beta_j p_{j-1}(\lambda),
\qquad (j=0,\dots,n-2),
\end{equation}
where $p_{-1}(\lambda) \equiv 0$, $\alpha_{j+1}
=\inprod{\lambda\overline{p_j},p_j}$ and
$\beta_{j+1} = \inprod{\lambda \overline{p_j},p_{j+1}}$.
Observe that $\beta_{j+1} > 0$ since the leading term of
$p_{j+1}$ has positive coefficient, a consequence
of executing  the Gram-Schmidt process. 
Proposition \ref{zeroprop}, the $L^2$ space 
$(\mathcal{P}_n(r2),\inprod{\cdot,\cdot})$
is $n$ dimensional.
Therefore $\lambda \overline{p_{n-1}(\lambda)}$
is a linear combination of $p_0,\dots,p_{n-1}$  
%By orthogonality,
%and denoting the right-hand side of \eqref{psrecurs} by $P_n(\lambda)$
%with $j=n-1$, we have $\inprod{P_n,P_n}=0$. Hence
%$P_n(\lambda_k)=0$ for all $k=1,\dots,n$.
and hence $\beta_n=0$.
In this way
we have associated these so-called called Jacobi parameters
$\alpha_1,\dots,\alpha_n \in\C$ and
$\beta_1,\dots,\beta_{n-1} > 0$ with the given biradial measure $\rho$.

To express this linear algebraically, let $Q=(q_{kj}) \in \Cnxn$ be the unitary matrix with columns $j$ defined by
\begin{equation}\label{Qpolyn}
q_{kj} = \sqrt{\rho_k} p_{j-1}(\lambda_k), \qquad (k=1,\dots,n).
\end{equation}
Then, by using the recursion
\eqref{psrecurs}, we get
\begin{equation}\label{Jdecomp}
\alp{D}\overline{Q} = Q\alp{J},
\end{equation}
where
\begin{align}\label{jtri}
\alp J &= \left[ \begin{array}{ccccc} 
\alpha_1&\beta_1&0&\cdots&0\\
\beta_1&\alpha_2&\ddots&\cdots&0 \\
\vdots&\ddots&\ddots& \ddots&\vdots \\
0&\ldots&&\alpha_{n-1}&\beta_{n-1}\\
0&\ldots&0&\beta_{n-1}&\alpha_n
\end{array} \right], \\
\label{jdiag}
\alp D &= \diag(\lambda_1, \dots, \lambda_n).
\end{align}
In this way, given a biradial measure with finite support, we have
the corresponding complex Jacobi matrix \eqref{jtri}.

Suppose, conversely,
that we are given a complex Jacobi matrix \eqref{jtri} and we want
to find a corresponding biradial measure. We start with the following
proposition.
\begin{proposition}
Let $\alpha_1,\dots,\alpha_n\in\C$ and let $\beta_1,\dots,\beta_{n-1} > 0$.
Then there exist numbers $\lambda_1,\dots,\lambda_n\in\C$ and
a unitary matrix $Q\in\Cnxn$ with positive first column
such that the equation \eqref{Jdecomp}
is satisfied, when $\alp J$ and $\alp D$ are
defined by formulae \eqref{jtri} and \eqref{jdiag}. Furthermore,
the numbers $\lambda_1,\dots,\lambda_n$ satisfy the property (ii) in
Definition \ref{finitemeas}.
\end{proposition}

\begin{proof}
By \cite[Corollary 4.4.4]{HJ1}, there exists a unitary matrix
$U\in\Cnxn$ and a nonnegative diagonal matrix
$\alp\Sigma = \diag(\sigma_1,\dots,\sigma_n)$
such that
\begin{equation}\label{ujsu}
\alp\Sigma \overline{U} = U\alp J.
\end{equation}
We show that no three of $\sigma_1,\dots,\sigma_n$ can be equal.
To reach a contradiction, we assume, without loss of generality, that
$\sigma_1=\sigma_2=\sigma_3$. Let $V\in\C^{3\times n}$ be the first
three rows of $U$ and let $v_1,\dots,v_n$ be the columns of $V$. By
\eqref{ujsu}, we have the recursion
\begin{align*}
\sigma_1 \overline{v_1} &= \beta_1 v_2 + \alpha_1 v_1, \\
\sigma_1 \overline{v_{j+1}} &= \beta_{j+1} v_{j+2} + \alpha_{j+1} v_{j+1}
+ \beta_j v_j,\qquad j=1,\dots,n-2.
\end{align*}
It follows that each $v_j$ is a linear combination of $v_1$ and
$\overline{v_1}$. This implies that the rank of $V$ is at most two,
contradicting the fact that the rank is three.

Let $\Theta\in\Rnxn$ be a diagonal matrix such that $Q=e^{i\Theta}U$ has
a nonnegative first column. Define
$\alp D = e^{2i\Theta}\alp\Sigma$. Then the identity
\eqref{Jdecomp} holds. That all $\lambda_1,\dots,\lambda_n$
are distinct in \eqref{jdiag} can be shown by an argument similar to the
one given above. Furthermore, if $q_j$ is a column of $Q^*$, then 
$\alp J\overline{q_j} = \lambda_j q_j$ and it is easy to see that the first
entry of $q_j$ must be nonzero.
\end{proof}

\begin{corollary}\label{corfinitesurj}
Let $\alpha_1,\dots,\alpha_n\in\C$ and let $\beta_1,\dots,\beta_{n-1} > 0$.
Then there exists a biradial measure $\mu$ with support of cardinality $n$
such that the complex Jacobi matrix corresponding to $\mu$ is given
by \eqref{jtri}.
\end{corollary}

\begin{proof}
Let $Q = (q_{kj})\in\Cnxn$ and $\alp D \in\Cnxn$ be according to the previous
proposition and let $\rho_k = q_{k1}^2$. Define the measure
$\rho = \sum_{k=1}^n \rho_k \delta_{\lambda_k}$. For each $j=1,\dots,n$
we then regard \eqref{Qpolyn} as a system of equations with $p_{j-1}$ as
the unknown.
These systems have invertible Vandermonde-type matrices,
by Proposition \ref{zeroprop}, and are therefore uniquely solvable.
A glance at the equation \eqref{Jdecomp} then shows that the polynomials
$p_{j-1}$ satisfy the recursion \eqref{psrecurs} with the same
coefficients as the given complex Jacobi matrix $\alp J$.
\end{proof}

To sum up, based on constructing an $L^2$ theory
for $\mathcal{C}(r2)$ and then orthogonalizing polynomials \eqref{polap}
according to \eqref{psrecurs}, 
we have  shown that the mapping from biradial 
measures with finite support
to  complex Jacobi matrices 
of type given in Corollary \ref{corfinitesurj}
is surjective. It is not injective, however.
This lack of injectivity will be completely described 
in Section \ref{noninjfinite}, 
after establishing surjectivity in the infinite dimensional case.

\subsection{Biradial measures with infinite support and antilinear 
Jacobi operators on infinite dimensional spaces}\label{infjacobisec}

We now turn to biradial measures with infinite supports in preparation
to the study of antilinear Jacobi operators on infinite dimensional
Hilbert spaces. We denote $\R^+=[0,\infty)$.

\begin{definition}
A Borel probability measure $\rho$ on $\C$ is said to have finite moments,
if it satisfies
\begin{equation}
\int_\C |\lambda|^n\,d\rho(\lambda) < \infty\qquad \text{for all }n=0,1,2,\dots.
\end{equation}
\end{definition}

\begin{definition}\label{genmeas}
Let $\rho_1^+, \rho_2^+$ be positive Borel measures on $\R^+$
such that
\[\rho_1^+(\R^+) + \rho_2^+(\R^+)=1.\]
Let $\phi_1,\phi_2$ be real-valued Borel-measurable functions on $\R^+$
and define the measurable transformation $R_{\phi_j}:\C \to \C$ by
\[R_{\phi_j}(z) = e^{-i \phi_j(|z|)} z.
\]
Define the positive Borel measures $\rho_j$ on $\C$ by
\[\rho_j(E) = \rho_j^+(R_{\phi_j}(E) \cap (\R^+\times \{0\})),
\qquad (j=1,2)\]
where $E\subset\C$ is Borel-measurable. Then $\rho = \rho_1 + \rho_2$
is called a biradial measure. Moreover, it is called symmetric
if $\phi_1(r) = \phi_2(r) + \pi$ (modulo $2\pi$).
\end{definition}

%% A biradial measure is called symmetric if $\phi_1(r) = \phi_2(r) + \pi$
%% (modulo $2\pi$).
%% \begin{definition}\label{symmeas}
%% Let $\rho$ be a Borel probability measure on $\R$,
%% and let $\phi$ be a real-valued Borel-measurable function on $\R^+$.
%% Define the measurable transformation $R_{\phi}:\C \to \C$ by
%% \[R_{\phi}(z) = e^{-i \phi(|z|)} z.
%% \]
%% Define the positive Borel measure $\mu$ on $\C$ by
%% \[\mu(E) = \rho\left(R_{\phi}(E) \cap (\R \times \{0\})\right),
%% \]
%% where $E\subset\C$ is Borel-measurable. Then $\mu$ is called
%% a symmetric biradial measure.
%% \end{definition}

We can represent a biradial measure $\rho$ in an equivalent form as follows.
Let $\rho^+ = \rho_1^+ + \rho_2^+$ and let $a_1,a_2$ be the Radon-Nikodym
derivatives $a_j=d\rho_j^+/d\rho^+ \,(j=1,2)$.
Note that $0\leq a_1(r),a_2(r) \leq 1$ and
$a_1(r) + a_2(r) = 1$ for almost every $r$.
Suppose $E\subset\C$ is measurable and write as follows
(where we use $R_{\phi_j}(E)$ as short-hand for
$R_{\phi_j}(E) \cap (\R^+\times \{0\})$ with the vacuous second
coordinate removed)
\begin{align*}
\rho_j(E) &= \int_0^\infty \rchi_{R_{\phi_j}(E)}\,d\rho_j^+
= \int_0^\infty a_j(r) \delta_r(R_{\phi_j}(E))\,d\rho^+(r) \\
&= \int_0^\infty a_j(r) \delta_{r e^{i\phi_j(r)}}(E)\,d\rho^+(r).
\qquad (j=1,2)
\end{align*}
Hence we can write
\begin{equation}\label{birmdis}
\rho(E) = \int_0^\infty \mu_r(E)\,d\rho^+(r),
\quad\text{where}\quad
\mu_r(E) = \sum_{j=1}^2 a_j(r) \delta_{r e^{i\phi_j(r)}}(E).
\end{equation}
The formula \eqref{birmdis} is the disintegration of the measure $\rho$
with respect to the measurable map $T:\C\to\R^+,\, T(z)=|z|$,
and $\rho^+$ \cite{CHP}.
Conversely, we could define a measure by the formula
\eqref{birmdis} and rewrite
it in the form of Definition \ref{genmeas}.

\smallskip

\begin{example}
Let $\mu$ be a positive Borel measure on $\R$. Define the measures
$\rho_1^+,\rho_2^+$ on $\R^+$ by $\rho_1^+(E)=\mu(E)$ and
$\rho_2^+(E)=\mu(-E \cap (-\infty,0))$, where $E\subset\R^+$.
Let $\phi_1 \equiv 0$ and
$\phi_2 \equiv \pi$. Then the measure $\rho$ on $\C$ in Definition
\ref{genmeas}
is supported inside $\R\times\{0\}$, and $\rho(F\times\{0\}) = \mu(F)$
for all Borel sets $F\subset\R$.
\end{example}

\smallskip

\begin{example} Let $\rho_1^+$ be a Borel probability measure with
support $\R^+$ and absolutely continuous with respect to the
Lebesgue measure.
Let $\{A_k\}_{k=1}^\infty$ be a partition
of $\R^+$ into metrically dense subsets, i.e., for all $k$ we have
$A_k$ Borel measurable and $\rho_1^+(A_k \cap (x-\epsilon,x+\epsilon)) > 0$
for all $x\in\R^+, \epsilon >0$ (e.g. \cite{ERO}). Let
$\{q_k\}_{k=1}^\infty$ be an enumeration of the rationals and
define $\phi_1 = \sum_{k=1}^\infty q_k \rchi_{A_k}$. Let $\rho_2^+=0$.
Then the measure $\rho$ in Definition \ref{genmeas} satisfies
$\rho(B(z,r)) > 0$ for every open disc $B(z,r) \subset \C$ and
therefore the support of $\rho$ is $\C$.
\end{example}

\smallskip

With respect to the 
 monomials $\pow{\lambda}{k}$, measures on $\setC$ 
can be reduced in dimension as follows.

\begin{lemma}\label{birmlemma}
Let $\mu$ be a Borel (probability) measure on $\C$ with finite moments.
Then there exists a symmetric biradial measure $\rho$ such that
\begin{equation}\label{birmequiv}
\int_\C p\,d\mu = \int_\C p\,d\rho \qquad\text{for all } p
\in\mathcal{P}(r2).
\end{equation}
Moreover, if $\mu$ is compactly supported, then also $\rho$ is.
\end{lemma}

\begin{proof}
Define the measurable map $T:\C\to\R^+,\, T(z)=|z|$ and let
$\{\mu_r\}_{r\in\R^+}$ be the disintegration of $\mu$ with respect
to $T$ and the image measure $\rho^+=\mu T^{-1}$  \cite{CHP}. For
nonnegative integer $k$, we then have
\begin{align*}
\int_\C \pow{\lambda}{k} d\mu(\lambda) = \int_0^\infty
\int_{|\lambda|=r} \pow{\lambda}{k}\,d\mu_r(\lambda)\,d\rho^+(r),
\end{align*}
where, for almost all $r$, $\mu_r$ is a Borel probability measure supported
in $\{|\lambda|=r\}$. For odd $k$, write $k=2l+1$, and then we get
\[\int_\C \pow{\lambda}{k} d\mu(\lambda) = \int_0^\infty
\int_{|\lambda|=r} \lambda\,d\mu_r(\lambda)\, r^{2l}\,d\rho^+(r),
\]
and we denote the inner integral (center of mass) on the right-hand side
by $C(r)$. Then we can choose $a_1(r),a_2(r)$ such that
$0\leq a_1(r),a_2(r)\leq 1$ and $a_1(r)+a_2(r)=1$, together with
$\phi_1(r),\phi_2(r)$, such that
\[C(r) = \sum_{j=1}^2 a_j(r) re^{i\phi_j(r)}.\]
If $C(r)\not=0$ is on the circle $\{|\lambda|=r\}$,
we can set $a_2(r)=0$ and then
$\phi_1(r)$ is uniquely determined (modulo $2\pi$). If $C(r)\not=0$ is
inside the circle, these choices are non-unique as illustrated in Figure
\ref{nonuniqfig}, and we let
$\phi_2(r)=\phi_1(r) + \pi$ (modulo $2\pi$) to get symmetry.
If $C(r)=0$, we let $\phi_1(r)=0,\,\phi_2(r)=\pi$.
We now define the measures
\[\nu_r = \sum_{j=1}^2 a_j(r) \delta_{re^{i\phi_j(r)}}.
\]
(Note that if $C(r)\not=0$, our choices for $a_1(r),a_2(r),\phi_1(r)$
and $\phi_2(r)$ make $\nu_r$ the unique measure of this form.)
We observe that
\[\int_{|\lambda|=r} \lambda\,d\mu_r(\lambda)
= C(r)
= \int_{|\lambda|=r} \lambda\,d\nu_r(\lambda).
\]
The measure defined by $\rho(E) = \int_0^\infty \nu_r(E)\,d\rho^+(r)$
is symmetric biradial, see formula \eqref{birmdis}, and we now have
\begin{equation}\label{birmequivm}
\int_\C \pow{\lambda}{k} d\mu(\lambda)
= \int_\C \pow{\lambda}{k} d\rho(\lambda)
\end{equation}
for all odd $k$. It is easy to see that \eqref{birmequivm} is true
for even $k$ as well, and then \eqref{birmequiv} follows.

The final claim follows from the fact that $\rho^+$ is compactly
supported if $\mu$ is.
\end{proof}

\begin{proposition}\label{polydenselemma}
Let $\mu$ be a compactly supported biradial measure. Then
the set of polynomials $\mathcal{P}(r2)$ is dense in $L^2(\mu)$.
\end{proposition}

\begin{proof}
Let $R>0$ be such that the support of $\mu$ is contained in an origin-centred
closed disc of radius $R$. We denote the disintegration of $\mu$
as in formula \eqref{birmdis}, and denote $\lambda_1(r) = 
re^{i\phi_1(r)},\,\lambda_2(r)=re^{i\phi_2(r)}$.

Take $f\in L^2(\mu)$ and $\epsilon>0$.
Let $g$ be a compactly supported smooth function on
$\C$ such that $\norm{f-g}_{L^2(\mu)}^2 < \epsilon/2$.
For $r$ such that $\lambda_1(r)=\lambda_2(r)$, let
$u(r^2)=g(\lambda_1(r))$ and $v(r^2)=0$, and otherwise
let $u(r^2)$ and $v(r^2)$ be the unique numbers such that
\[g(\lambda_j(r)) = u(r^2) + v(r^2)\lambda_j(r),\qquad(j=1,2).\]
For $r$ such that $\lambda_1(r)\not=\lambda_2(r)$, we now have
\[\frac{g(\lambda_1(r)) - g(\lambda_2(r))}{\lambda_1(r)-\lambda_2(r)}
= v(r^2),\]
where the left hand side is bounded. Hence $v$ is a bounded function
and it follows that $u$ is bounded as well.
Let $p$ and $q$ be ordinary polynomials such that
\[\int_0^R |u(r^2)-p(r^2)|^2\,d\rho^+(r) < \frac{\epsilon}{8}
\quad\text{and}\quad
\int_0^R |v(r^2)r-q(r^2)r|^2\,d\rho^+(r) < \frac{\epsilon}{8}.\]
Then we have
\[\int_{\C} \left|f(\lambda) - p(|\lambda|^2) - q(|\lambda|^2)\lambda
\right|^2\,d\mu(\lambda) < \epsilon.\]
\end{proof}

%\subsection{Jacobi operators on infinite dimensional spaces}\label{infjacobisec}

With these measure theoretic preparations, let $\rho$ be a Borel (probability) measure on $\C$ with finite moments.
On $\mathcal{P}(r2)$, we define the inner product
\begin{equation}\label{innerpa}
\inprod{p,q} = \int_{\C} p\overline{q}\,d\rho.
\end{equation}
We then apply the Gram-Schmidt process to
the monomials $1,$ $\pow{\lambda}{1},$ $\pow{\lambda}{2},$ $\dots$
with respect to
the inner product \eqref{innerp} and obtain an orthonormal sequence
of polynomials $p_0,p_1,p_2,\dots$, where $p_j$ has degree $j$.
The process breaks down if and only if
$\dim \mathcal{P}_n(r2) \leq n$ for some $n$.
In the breakdown case, for the least such $n$, we have the orthonormal
polynomials $p_0,\dots,p_{n-1}$ and the corresponding Jacobi parameters
$\{\alpha_j\}_{j=1}^n$, $\{\beta_j\}_{j=1}^{n-1}$ similar to the
case of biradial measures with finite support.
If the process
does not break down, we get infinitely many Jacobi parameters
$\{\alpha_j\}_{j=1}^\infty$ with positive $\{\beta_j\}_{j=1}^\infty$ and the
recursion \eqref{psrecurs} holds for all $j=0,1,2,\dots$.
These Jacobi parameters are recorded in the infinite matrix
\begin{equation}\label{infjacobi}
\alp{J} = \begin{bmatrix}
\alpha_1&\beta_1&0&\cdots \\
\beta_1&\alpha_2&\beta_2 & \ddots \\
0&\beta_2&\alpha_3&\ddots  \\
\vdots&\ddots&\ddots&\ddots  \\
\end{bmatrix}.
\end{equation}

\begin{proposition}\label{propsamemoments}
Let $\rho,\rho'$ be Borel (probability) measures with finite moments
and at least $n$ Jacobi parameters,
$\{\alpha_j\}_{j=1}^n$, $\{\beta_j\}_{j=1}^{n-1}$
and
$\{\alpha'_j\}_{j=1}^n$, $\{\beta'_j\}_{j=1}^{n-1}$,
respectively. Then $\alpha_j=\alpha_j'\quad(j=1,\dots,n)$ and
$\beta_j=\beta_j'\quad(j=1,\dots,n-1)$ if and only if
\begin{equation}\label{samemoments}
\int_\C \pow{\lambda}{k}\,d\rho =
\int_\C \pow{\lambda}{k}\,d\rho' \qquad \text{for all } k=0,1,\dots,2n-1.
\end{equation}
\end{proposition}

\begin{proof}
The proof for the real Jacobi parameter case carries over easily
\cite[Proposition 1.3.4]{SIBOOK}.
\end{proof}

The following is a version of Favard's theorem for bounded antilinear
Jacobi operators. We prove it similarly as in \cite{SIBOOK}.

\begin{theorem}\label{boundedmu}
Let $\{\alpha_j\}_{j=1}^\infty$, $\{\beta_j\}_{j=1}^\infty$
be bounded Jacobi parameters. Then there exists a compactly supported
symmetric biradial measure $\mu$ on $\C$ such that the Jacobi parameters
corresponding to $\mu$ are
$\{\alpha_j\}_{j=1}^\infty$, $\{\beta_j\}_{j=1}^\infty$.
\end{theorem}

\begin{proof}
By Corollary \ref{corfinitesurj}, for each $n$ there exists
a biradial measure $\mu_n$ such that the corresponding Jacobi
parameters are $\{\alpha_j\}_{j=1}^n$, $\{\beta_j\}_{j=1}^{n-1}$,
and we denote the corresponding complex Jacobi matrix by $\alp{J^{[n]}}$.
Since the $n$ points supporting $\mu_n$ are coneigenvalues of
$\alp{J^{[n]}}$, the support of $\mu_n$ is contained in the closed
disc $\overline{B}(0,\norm{\alp{J^{[n]}}}) \subset
\overline{B}(0,\norm{\alp J\tau}) =: B$. Moreover, $\mu_n(\C) = 1$ is bounded
for all $n$ and
therefore there exists a weakly converging subsequence
$\{\mu_{n_j}\}_{j=1}^\infty$ such that 
\[
\lim_{j\to\infty} \int_{B} f\,d\mu_{n_j} =
\int_{B} f\,d\mu
\]
for all continuous functions $f$ on $B$, where $\mu$ is a Borel
probability measure supported in $B$. For each nonnegative integer $k$,
we choose
$f(\lambda) = \pow{\lambda}{k}$ and use Proposition \ref{propsamemoments}
to conclude that the Jacobi parameters of $\mu$ are
$\{\alpha_j\}_{j=1}^\infty$, $\{\beta_j\}_{j=1}^\infty$. Moreover,
an application of Lemma \ref{birmlemma} replaces $\mu$ with a
symmetric biradial measure.
\end{proof}

We next obtain a spectral theorem for bounded antilinear self-adjoint
operators. For an approach using spectral integrals, see \cite{Santtu}.

\begin{theorem}\label{spectralthm}
Let $A$ be a bounded antilinear self-adjoint operator on a Hilbert
space $H$. Suppose there exists a cyclic vector of $A$. Then there
exists a compactly supported symmetric biradial measure $\mu$ on $\C$,
and a $\C$-linear isometric isomorphism $U:H\to L^2(\mu)$, such that
\begin{equation}\label{antimulti}
UAU^{-1} = \lambda \tau,
\end{equation}
%\[UAU^{-1} = \lambda \tau,\]
where $\lambda\tau$ denotes the multiplication operator
$f(\lambda) \mapsto \lambda\overline{f(\lambda)}$ on $L^2(\mu)$.
\end{theorem}

\begin{proof}
Let $x\in H$ be the cyclic vector and let $\{q_j\}_{j=1}^\infty$ be
the orthonormal basis of $H$ generated by the Arnoldi process starting
from the vector $x$. Moreover, let $\alp{J}$ be the generated complex (infinite)
Jacobi matrix with parameters
$\{\alpha_j\}_{j=1}^\infty$, $\{\beta_j\}_{j=1}^\infty$.
By Theorem \ref{boundedmu} the measure $\mu$ exists, with
orthonormal polynomials $\{p_j\}_{j=0}^\infty$, such that
\[ \beta_j p_{j-1}(\lambda) + \alpha_{j+1}p_j(\lambda) +
 \beta_{j+1} p_{j+1}(\lambda) = \lambda\overline{p_j(\lambda)}
\qquad(j=0,1,2,\dots),
\]
where $p_{-1}\equiv 0$. The set of polynomials $\{p_j\}_{j=0}^\infty$ is
an orthonormal basis of $L^2(\mu)$ by Proposition \ref{polydenselemma}.
We define the isometric isomorphism  $U$ by $U(q_j) = p_{j-1}\,(j=1,2,\dots)$
and the claim follows.
\end{proof}

\begin{corollary}\label{spectralthmf}
Under the assumptions of Theorem \ref{spectralthm}, for any
continuous biradial function $\hat{f}$ defined on $\sigma(A)$,
\[Uf(A)U^{-1} = u(|\lambda|^2) + v(|\lambda|^2)\lambda \tau,\]
where $u$ and $v$ are as in \eqref{biradf}.
\end{corollary}

\begin{proof}
Apply Lemma \ref{sapolyapprox}.
\end{proof}

\begin{proposition}\label{spectrumsupprop}
Biradial measures $\mu$ given by Theorem \ref{spectralthm}
are supported in the spectrum $\sigma(A)$.
\end{proposition}

\begin{proof}
By Theorem \ref{spectralthm}, we have $UA^2U^{-1} = |\lambda|^2$ and
therefore
\begin{equation}\label{spectA2}
Up(A^2)U^{-1} = p(|\lambda|^2)
\end{equation}
for all ordinary complex analytic polynomials $p$. 
The operator $A^2$ is $\C$-linear
self-adjoint and positive semidefinite.

Let $\lambda_0\not\in \sigma(A)$.
Let $g(\lambda) = f(|\lambda|^2)$, where
$f$ is a nonnegative compactly supported function on $\R^+$
such that $f$ vanishes
on $\sigma(A^2)$  and $g(\lambda_0) = 1$, which is possible due to
Proposition \ref{HSprop}. Let $\{p_j\}$ be a sequence
of ordinary polynomials converging uniformly to $f$ on the support of $f$.
Then $\norm{f(A^2) - p_j(A^2)} \to 0$ as $j\to\infty$, and, since $f(A^2)=0$,
from \eqref{spectA2} we have
\[\int_\C g(\lambda)\,d\mu(\lambda)=0.\]
Hence there exists an open set $V\subset \C$ containing $\lambda_0$
such that $\mu(V)=0$ and therefore $\lambda_0$ is not in the support of $\mu$.
\end{proof}

We are now ready to prove Theorem \ref{spectmapthm2}
and Corollary \ref{spectradprop}.

\begin{proof} {\it of Theorem \ref{spectmapthm2}}.
The inclusion $\subset$ for \eqref{spectmapeq} is Theorem \ref{spectmapthm1}.
For the inclusion $\supset$,
take $\lambda_0 \not\in \hat{f}(\sigma(A))$,
%Represent $f$ as
%\begin{equation}\label{ftaurepr}
%f(\lambda\tau) = u(|\lambda|^2) + v(|\lambda|^2)\lambda\tau,
%\end{equation}
and let $H$ and $A$ be decomposed as in \eqref{hdirectsum} and
\eqref{adirectsum}.
By Corollary \ref{spectralthmf}, we have
\begin{equation}\label{fmultop}
U(\lambda_0 I_\alpha - f(A_\alpha))U^{-1} =
\lambda_0 - u(|\lambda|^2) - v(|\lambda|^2)\lambda\tau.
\end{equation}
Since $\sigma(A)$ is a compact set,
\[c := \min_{\lambda\in\sigma(A)}
\Big|\left|\lambda_0 - u(|\lambda|^2)\right| -
\left|v(|\lambda|^2)\lambda\right|\Big|
=\min_{\lambda\in\sigma(A)}
|\lambda_0 - \hat{f}(\lambda)| > 0,\]
and, since $\hat{f}(\sigma(A)) \supset \hat{f}(\sigma(A_\alpha))$
by \eqref{spectcup}, we have
\[ 0 < c \leq \min_{\lambda\in\sigma(A_\alpha)}
\Big|\left|\lambda_0 - u(|\lambda|^2)\right| -
\left|v(|\lambda|^2)\lambda\right|\Big|
\qquad\text{for all }\alpha.\]
Hence the multiplication operator on the right-hand side of 
equation \eqref{fmultop} is boundedly invertible, with the inverse
having the norm estimate
%% \[\max_{\lambda\in\sigma(A_\alpha)}
%% \frac{\left|\lambda_0 - u(|\lambda|^2)\right| +
%% \left|v(|\lambda|^2)\lambda\right|}
%% {\Big|\left|\lambda_0 - u(|\lambda|^2)\right|^2 -
%% \left|v(|\lambda|^2)\lambda\right|^2\Big|}
%% \leq \]
\[\max_{\lambda\in\sigma(A_\alpha)}
\frac{1}{\Big|\left|\lambda_0 - u(|\lambda|^2)\right| -
\left|v(|\lambda|^2)\lambda\right|\Big|} \leq \frac{1}{c} < \infty
\qquad\text{for all }\alpha.
\]
We have shown that
the operators $\lambda_0 I_\alpha - f(A_\alpha)$ are boundedly
invertible for all $\alpha$, with the norms of their
inverses uniformly bounded.
Hence $\lambda_0 I - f(A)$ is boundedly invertible, so
$\lambda_0\not\in \sigma(f(A))$.
\end{proof}

\begin{proof} {\it of Corollary \ref{spectradprop}}.
The second equality follows from Theorem \ref{spectmapthm2}.
For the first, let $H$ and $A$ be decomposed as in \eqref{hdirectsum} and
\eqref{adirectsum}, and
note that $\norm{f(A)} = \sup_\alpha \norm{f(A_\alpha)}$. Then, by
Corollary \ref{spectralthmf},
\begin{align*}
\norm{f(A_\alpha)g}^2 &\leq \int_{\sigma(A_\alpha)}
(|u(|\lambda|^2)| + |v(|\lambda|^2)\lambda|)^2
|h(\lambda)|^2\,d\mu(\lambda) \\
&\leq \sup_{\lambda\in\sigma(A_\alpha)} |\hat{f}(\lambda)|^2
\norm{h}_{L^2(\mu)}^2,
\end{align*}
where we
%represented $f$ as in \eqref{ftaurepr} and
denoted $h = Ug$.
Hence $\norm{f(A_\alpha)} \leq \sup_{\lambda\in\sigma(A_\alpha)}
|\hat{f}(\lambda)|$. The inequality in the other direction is established
by considering invertibility of $\lambda I_\alpha - f(A_\alpha)
= \lambda(I_\alpha - \lambda^{-1}f(A_\alpha))$ via the Neumann series
expansion 
in the usual way.
\end{proof}

Starting from \cite{BEU}, 
multiplication operators have played a major role in
understanding the $\setC$-linear
invariant subspace problem. 
For the corresponding problems
in the antilinear case,
the role of multiplication is taken by the operator
\eqref{antimulti}.

\section{Unbounded antilinear Jacobi operators and the moment problem}\label{sec4}
Let $\rho$ be a biradial measure with finite moments and define
\begin{equation}\label{infmoments}
m_k = \int_\C \pow{\lambda}{k}\,d\rho(\lambda),\qquad
k=0,1,2,\dots .
\end{equation}
In the corresponding 
moment problem we are given an arbitrary sequence of complex numbers
$\{m_k\}_{k=1}^\infty$ such that, after a possible scaling, $m_0=1$.
Of course, then necessarily $m_{2k}\geq 0$ for $k\in \setN$.
The problem consists of finding a biradial measure
$\rho$ such that the equation \eqref{infmoments} is satisfied. The
answer %rests on a property of the following matrix
can be given in terms of the matrix 
\begin{equation}\label{momentmat}
M = \begin{bmatrix}
m_0 & \overline{m_1} & m_2 & \overline{m_3} & \cdots \\
m_1 & m_2 & m_3 & m_4 & \cdots \\
m_2 & \overline{m_3} & m_4 & \overline{m_5} & \cdots \\
m_3 & m_4 & m_5 & m_6 & \cdots \\
\vdots & \vdots & \vdots & \vdots & \ddots
\end{bmatrix},
\end{equation}
where, denoting again the conjugation operator by $\tau$,
$M_{ij} = \tau^{j-1} m_{i+j-2}$.  Observe that $M$ is 
Hermitian, as well as, in a certain
sense, a Hankel-like matrix. Thereby it is natural to 
contrast this with the classical Hamburger moment problem.

\begin{theorem}\label{momentthm}
Let $\{m_k\}_{k=1}^\infty \subset \C$ and $m_0=1$. Then there exists
a biradial measure $\rho$ with finite moments such that
\eqref{infmoments} holds if and only if
the submatrix $M_{1:k,1:k}$ of \eqref{momentmat} is positive semi-definite
for all $k$.
\end{theorem}

\begin{proof}
That the claim holds with a Borel probability measure $\mu$,
in lieu of a biradial measure $\rho$, follows from \cite[Theorem 1]{STSZ}.
This measure $\mu$ can then
be replaced by a biradial measure $\rho$ by Lemma \ref{birmlemma}.
\end{proof}

\begin{corollary}
The mapping from biradial measures to complex Jacobi parameters
is surjective.
\end{corollary}

\begin{proof}
Let $\{\alpha_j,\beta_j\}_{j=1}^\infty \in (\C\times(0,\infty))^\infty$
be complex Jacobi parameters. Define a sequence of
polynomials $\{p_j\}_{j=0}^\infty$ by the recursion \eqref{psrecurs} starting
with $p_0\equiv 1$. Define an inner product in $\mathcal{P}(r2)$ by
\[\inprod{p,q} = \sum_{j\geq 0} a_j \overline{b_j},
\qquad (p,q\in\mathcal{P}(r2)),
\]
where $p = \sum_j a_j p_j$ and $q = \sum_j b_j p_j$, and let
\[m_k = \inprod{\pow{\lambda}{k},1} \qquad \text{for }k=0,1,2,\dots.
\]
From the recursion \eqref{psrecurs} we see
that multiplication by $\lambda\tau$ in $\mathcal{P}(r2)$ is
represented by $\alp J\tau$ in the basis $\{p_j\}_{j=0}^\infty$, where
$\alp J$ is the infinite complex Jacobi matrix given by formula
\eqref{infjacobi}. From this we find
\[m_{i+j} = e_1^T (\alp J\tau)^{i+j} e_1 =
\tau^j \left( (\alp J\tau)^j e_1 \right)^* \left( (\alp J\tau)^i e_1 \right) =
\tau^j \inprod{\pow{\lambda}{i},\pow{\lambda}{j}},
\]
where $e_1 = \begin{bmatrix} 1 & 0 & 0 & \cdots \end{bmatrix}^T$.
Since the inner product is positive-definite, the condition of
Theorem \ref{momentthm} is satisfied and therefore there exists
a biradial measure $\rho$ such that
\[\inprod{\pow{\lambda}{k},1} = m_k = \int_{\C} \pow{\lambda}{k} d\rho
\qquad \text{for all }k.\]
Since $\tau^j \pow{\lambda}{i+j} = \pow{\lambda}{i}
\overline{\pow{\lambda}{j}}$, we get $\inprod{p,q} = \int_\C p\overline{q}\,
d\rho$ for any polynomials $p$ and $q$. This completes the proof.
\end{proof}

\section{Noninjectivity of the mapping of measures to Jacobi operators}
\label{noninjfinite}

The mapping of biradial measures to antilinear Jacobi operators is
surjective, but noninjective (even for measures with bounded support).
We do not discuss this in full generality, but content ourselves
with a precise characterization in the finite dimensional case
covered in Section \ref{bifisu}.

\begin{theorem}\label{noninjfinthm}
Let $\rho,\rho'$ be biradial measures with supports of cardinality $n$.
Then $\rho$ and $\rho'$ have the same Jacobi parameters if and only if
(using the notation of Definition \ref{finitemeas}) we have
\begin{align}
m&=m', \notag \\
\rho_k &= \rho_k',\qquad &&\text{when } 2m+1\leq k \leq n, \notag \\
\label{nonuniqcond}
\rho_{2k-1} + \rho_{2k} &= \rho_{2k-1}' + \rho_{2k}', \qquad
&&\text{when } 1\leq k \leq m,  \\
\lambda_k &= \lambda_k',\qquad &&\text{when } 2m+1\leq k \leq n, \notag \\
|\lambda_{2k}| &= |\lambda_{2k}'|,\qquad &&\text{when } 1\leq k \leq m,
\notag \\
\rho_{2k-1}\lambda_{2k-1} + \rho_{2k} \lambda_{2k}
&= \rho_{2k-1}'\lambda_{2k-1}' + \rho_{2k}'\lambda_{2k}',
\qquad &&\text{when } 1\leq k \leq m. \notag
\end{align}
\end{theorem}

\begin{proof}
It is easy to verify that the conditions \eqref{nonuniqcond} imply
\eqref{samemoments}. Hence sufficiency follows from Proposition
\ref{propsamemoments}.

To prove necessity, assume the condition \eqref{samemoments} holds.
First, the equation $\eqref{Jdecomp}$ implies that $\alp{J}\overline{\alp J}$
and $\alp{D}\overline{\alp D}$ have the same eigenvalues. Hence
the sets $\{|\lambda_i|\}_{i=1}^n$ and $\{|\lambda_i'|\}_{i=1}^n$ are
the same, including multiplicities, which implies the first and the fifth
condition in \eqref{nonuniqcond}.
Moreover, the conditions
given by \eqref{samemoments} for even $k$ then imply the second and
the third condition in \eqref{nonuniqcond}. Finally, for odd $k$,
\eqref{samemoments} imply the fourth and the sixth condition in
\eqref{nonuniqcond}.
\end{proof}

The preceding theorem says that the total mass on each
origin centred circle is unique and the centre of the masses on each
such circle is unique as well. If the centre of mass is on the circle, then
all the mass is concentrated on a single unique point. If the centre
of mass is inside the circle, then the mass is distributed
among two non-unique points as illustrated in Figure \ref{nonuniqfig}.

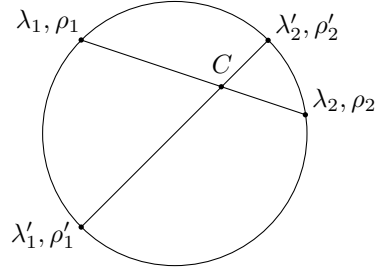
\begin{figure}
\begin{center}
\setlength{\unitlength}{.1pt}
\begin{picture}(1000,1000)
\put(500, 500){\circle{1000}}
\put(146,854){\line(12,-4){850}}
\put(146,854){\circle*{20}}
\put(-100,900){$\lambda_1, \rho_1$}
\put(995,571){\circle*{20}}
\put(1020,600){$\lambda_2, \rho_2$}
\put(677,677){\circle*{20}}
\put(640,730){$C$}
\put(146,146){\line(1,1){707}}
\put(146,146){\circle*{20}}
\put(-120,90){$\lambda_1', \rho_1'$}
\put(853,853){\circle*{20}}
\put(880,880){$\lambda_2', \rho_2'$}
\end{picture}
\caption{\label{nonuniqfig}  An illustration of the case where the centre of
mass $C$ resides inside
the circle. Here $\rho_i,\rho_i'$ are the masses of the points and
$\lambda_i,\lambda_i'$ are their positions. The total mass and the centre of
mass are unique.}
\end{center}
\end{figure}

There is also an immediate application of this result in numerical analysis.

%\Appendix
%\section{$\R$-linear GMRES}
\smallskip
\begin{example} Assume that $\alp{M}\in \setC^{n \times n}$ 
is a very large matrix and $b \in \setC^n$.
The $\R$-linear GMRES (Generalized Minimal Residual) method is an
iterative method for solving the linear system
$$\alp{M} \overline{x} = b.$$
In \cite{HP}  its convergence behaviour was
analyzed. If $\alp{M}$ complex symmetric, 
then it can be assumed that
$\alp{M}$ is actually a diagonal matrix. 
Suppose, moreover, that
$\alp{M}$ has distinct diagonal entries such that each
origin centred circle intersected either two of them or none, and that
and $b\in\R^n$ is a vector with all its entries ones. 
Then in \cite[Section 5]{HP}  it was observed that the
numerical convergence behaviour was such that the
residual dropped only at every other iteration step.

We can now explain this observation as follows.
The norm of the residual vector $r_k = b - \alp{M} \overline{x_k}$,
where $x_k$ is the approximation at the $k$th  step, satisfies
\[\norm{r_k} = \min_{\substack{p \in \mathcal{P}_k(r2)\\p(0)=1}}
\norm{p(\alp{M})b}
\leq 
\min_{\substack{p \in \mathcal{P}_{k}(r2) \\ p(0)=1}}
\max_{\lambda\in \sigma(\alp{M})}
\left| p(\lambda) \right| \norm{b}.\]
By Theorem \ref{noninjfinthm}, there exists a unitary matrix
$Q\in\Cnxn$ such that $\rho=Q^*b\in\R^n$ and that the matrix
$\alp D = Q^* \alp{M} \overline{Q}$
is a diagonal matrix with the property that if 
$\lambda\in\C$ appears as a diagonal
entry then also $-\lambda$ does. Now we have
\[\norm{r_k} = \min_{\substack{p \in \mathcal{P}_k(r2)\\p(0)=1}}
\norm{p(\alp D)\rho}
\leq 
\min_{\substack{p \in \mathcal{P}_{k}(r2) \\ p(0)=1}}
\max_{\lambda\in \sigma(\alp D)}
\left| p(\lambda) \right| \norm{\rho}.\]
If $p(\lambda)=u(|\lambda|^2) + v(|\lambda|^2)\lambda$ is the minimizing
polynomial to the problem on the right-most side, then
due to symmetry, also $u(|\lambda|^2) - v(|\lambda|^2)\lambda$ is. Hence,
by uniqueness, $v=0$ and $p(\lambda)=u(|\lambda|^2)$. Thereby we may expect
that the residual drops only at every other iteration step.
\end{example}

\end{document}